\documentclass[11pt]{amsart}
\author{Maya Saran}

\address{Ashoka University\\
Sonepat, Haryana\\
India}

\email{maya.saran@ashoka.edu.in}
\title{$G_{\delta}$ sets in $\sigma$-ideals generated by compact sets}
\usepackage{amsmath}
\usepackage{amsthm}
\usepackage{amssymb,latexsym}
\usepackage{enumerate}

\newtheorem{thm}{Theorem}

\newtheorem{cor}[thm]{Corollary}
\newtheorem{lemma}[thm]{Lemma}

\newtheorem{prop}[thm]{Proposition}

\theoremstyle{definition}

\newcommand{\F}{\mathcal{F}}
\newcommand{\Gee}{\mathcal{G}}

\newcommand{\B}{\mathcal{B}}
\newcommand{\A}{\mathcal{A}}
\newcommand{\U}{\mathcal{U}}
\newcommand{\V}{\mathcal{V}}
\newcommand{\W}{\mathcal{W}}
\newcommand{\K}{\mathcal{K}}
\newcommand{\gd}{G_{\delta}}
\newcommand{\fs}{F_{\sigma}}
\newcommand{\ol}[1]{\overline{#1}}
\newcommand{\bbn}{\mathbb{N}}

\begin{document}

\begin{abstract}
	Given a compact Polish space $E$ and   the hyperspace of its compact subsets $\mathcal{K}(E)$, we consider  $G_{\delta}$ $\sigma$-ideals of compact subsets of $E$.  Solecki has shown that any  $\sigma$-ideal in a broad natural class of $G_{\delta}$ ideals can be represented via a compact subset of $\mathcal{K}(E)$; in this article we examine the behaviour of $G_{\delta}$ subsets of $E$ with respect to the  representing set. Given an ideal $I$ in this class, we construct a representing set that recognises a compact subset of $E$ as being ``small'' precisely when it is in $I$, and recognises a $G_{\delta}$ subset of $E$ as being ``small'' precisely when it is covered by countably many compact sets from $I$. 
\end{abstract}
\maketitle
\section{Introduction}

\noindent Let $E$ be a compact Polish space and let $\K(E)$ denote the hyperspace of its compact subsets, equipped with the Vietoris topology. A basis for this topology consists  of sets of the  form 
\[\{F \in \K(E): F \subseteq U_0,\; F \cap U_i \neq \emptyset\; \forall i = 1,\ldots, k\}\] for some $k \in \bbn$ and open sets $U_0,\ldots,U_k \subseteq E$, with $U_i \subseteq U_0$ for each $i$. A nonempty set $I \subseteq \K(E)$ is an  \emph{ideal} of compact sets if it is closed under the operations of taking compact subsets and finite unions. An ideal $I$ is a \emph{$\sigma$-ideal} if it is also closed under countable unions whenever the union itself is compact. Such families arise commonly in analysis out of various notions of smallness. 

It is now long established that the condition of being an ideal or $\sigma$-ideal of compact sets is strongly related to the complexity of $I$ considered as a subset of $\K(E)$. By a result of Dougherty-Kechris and Louveau (see \cite{Kec2}), we know that if $I$ is a $\gd$ ideal, it must in fact be a $\sigma$-ideal, and by results of Kechris--Louveau--Woodin proved in the seminal paper \cite[section 1]{KLW}, we know that if a  $\sigma$-ideal is coanalytic or analytic, it must be either  complete coanalytic or simply $\gd$. Having established this dichotomy \cite{KLW} focused more on coanalytic ideals than on the class of $\gd$ ideals, however the latter includes rich examples and its theory is far from trivial. 

In this article we consider $\gd$ $\sigma$-ideals of compact sets that also satisfy the following condition, formulated by Solecki in  \cite{SS1}: a collection of compact sets $I \subseteq \K(E)$ has \emph{property~$(*)$} if, for any sequence of sets $(K_n)_{n \in \bbn} \subseteq I$, there exists a $\gd$ set $G$ such that $\bigcup_n K_n \subseteq G$ and $\K(G) \subseteq I$.  While property $(*)$ is stronger than the condition of being a $\sigma$-ideal, it holds in all natural examples of $\gd$ $\sigma$-ideals, including the ideals of compact meager sets, measure-zero sets, sets of dimension $\leq n$ for fixed $n \in \bbn$, and Z-sets. (See \cite{SS1} for these and other examples and a discussion of property $(*)$.) Solecki has shown  that any such ideal is  represented via the meager ideal of a closed subset of $\K(E)$, as follows. For $A \subseteq E$, we define $A^*$ as the set of all compact subsets of $E$ that intersect $A$, i.e., 
\[A^* = \{K \in \K (E) : K \cap A \neq \emptyset \}.\]  
The representation theorem (\cite[section 3]{SS1}) says that, if $I \subseteq \K(E)$ is coanalytic and nonempty, then $I$ has property $(*)$ if and only if there exists a closed set $\F \subseteq \K(E)$ such that 
\[\forall K \in \K(E),\; K \in I \iff K^* \textrm{ is meager in } \F.\]
Here when we say that $K^*$  is meager in $\F$, we mean that $K^* \cap \F$  is relatively meager in $\F$.
(This representation is a `category analogue' to a result of Choquet (see \cite{Choquet}) that establishes a correspondence between alternating capacities of order $\infty$  on $E$ and probability Borel measures on $\K(E)$.) We have shown in \cite{MS} that as long as $I$ has no nonmeager sets the set $\F$  can be assumed to be \emph{upward closed}, i.e., it contains all compact supersets of its members. The upward closedness of  $\F$ ensures that the map $K  \mapsto K^* \cap \F$, a fundamental function in this context, is continuous.

Motivated by  results on $\gd$ ideals with property $(*)$, Solecki has framed in \cite[section 7]{SS1} some natural further questions about the relationship of $\gd$ subsets of $E$ to a representing set $\F$. (Another question asks if every thin $\gd$ ideal has property $(*)$.)
The present paper takes up the following particular question: given a $\gd$ ideal $I$ with property $(*)$, does $I$ have a representing set  $\F$ for which, given a  $\gd$ set $G\subseteq E$,  $G^*$ will be meager in $\F$ precisely when $G$ is covered by countably many compact sets in $I$? 
We answer in the affirmative, and if the ideal under consideration contains only  sets with empty interiors, we construct a representing set $\F$ that is also upward closed. (The other questions remain open.)  

\begin{thm} \label{yay} 
	Let $E$ be a compact Polish space and let $I \subseteq \K(E)$ be a  $\gd$ ideal of compact sets, with property~$(*)$. Then there exists a  compact set $\F \subseteq \K(E)$ such that any $\gd$ subset $G$ of $E$ is covered by countably many sets in $I$ if and only if $G^*$ is meager in $\F$.
	
	Further, if  $I$ contains only meager sets, then there exists a set $\F$ as above that is also upward closed. 
	\end{thm}

	In fact a stronger statement is true:  we can replace $\gd$ subsets of $E$ by analytic sets  by invoking Solecki's theorem from \cite{SS2}, viz., given a family of compact subsets of a Polish space $E$, an analytic subset of $A$ of $E$ cannot be covered by countably many sets from the family if and only if $A$ has a $\gd$  subset that cannot be so covered. We give this stronger result as a corollary.

\begin{cor}  
	Let $E$ be a compact Polish space and let $I \subseteq \K(E)$ be a  $\gd$ ideal of compact sets, with property~$(*)$. Then there exists a  compact set $\F \subseteq \K(E)$ such that any analytic subset $A$ of $E$ is covered by countably many sets in $I$ if and only if $A^*$ is meager in $\F$.	

	Further, if  $I$ contains only meager sets, then there exists a set $\F$ as above that is also upward closed. 
	\end{cor}

Note that the set $\F$ (of  both the theorem and the corollary) will indeed be a representing set for $I$, i.e., for any compact  $K \subseteq E$, the membership of $K$ in $I$ will be characterized by the meagerness of $K^*$ in  $\F$. (The trivial proof of this is indicated in Proposition \ref{prop}.)

We remark that the condition that $I$  contain only meager sets is necessary if the set $\F$ of Theorem \ref{yay} is to be upward closed (indeed, if any representing set for $I$ is to be upward closed).  For if  $\F \subseteq \K(E)$ is nonempty and upward closed and  $U \subseteq E$ is nonempty and open, the set $U^* \cap \F$, an open subset of $\F$, is automatically nonempty. Thus if any set $K \subseteq E$ has nonempty interior, $K^*$ will have nonempty interior in $\F$.

A final remark on a consequence of Theorem \ref{yay}. In some cases, the notion of smallness that defines $I$ can also be applied in a natural way to $\gd$ sets. (For example, $I$ might consist of compact nullsets with respect to Lebesgue measure on $[0,1]$;  here the notion of being a nullset  applies to $\gd$ sets as well.)  Theorem \ref{yay} implies that a representing set $\F$ that works to identify small compact sets $K$ (via the meagerness of $K^*$ in $\F$) need not work to identify small $\gd$ sets. (Continuing with the same example, the theorem gives us a  set $\F$ such  that, if  $K \subseteq E$ is compact then $K^*$ is meager in $\F$ exactly when $K$ is a nullset --- but given a $\gd$ set $G$, the meagerness of $G^*$ in $\F$ depends not at all on whether $G$ is null, but only on whether $G$ is covered by countably many compact nullsets. And certainly $[0,1]$ does have $\gd$ nullsets that cannot be so covered, for example, comeager nullsets.)

\section{Preliminaries}

We will say that a set $\A \subseteq \K(E)$ is \emph{downward closed} if $\A$ contains all compact subsets of its members. We will use the notation $\gd(E)$ for the collection of all $\gd$ subsets of $E$. 

The next proposition indicates the approach we will take in proving Theorem \ref{yay}:  if  $\F\subseteq \K(E)$ is known to be a representing set for $I$, it suffices to consider $\gd$ sets all of whose nonempty open subsets have closure not in $I$. 

\begin{prop} \label{prop} 
Let $E$ be a compact Polish space and let $I \subseteq \K(E)$ be a nonempty $\gd$ ideal of compact sets, with property~$(*)$. Let  $\F  \subseteq \K(E)$. 
Then the following holds
\begin{multline}
\forall  G  \in  \gd(E),\\
    \Big( G \subseteq \bigcup_{n\in \bbn} K_n\ \text{ for some } K_n \in I \Big) \iff \Big(G^*  \textrm{ meager in } \F \Big) 
    \label{1}%
\end{multline}
if and only if the two conditions below hold:
\begin{equation}
\forall K \in \K(E),\; K \in I \iff K^* \textrm{ is meager in } \F,
\label{char}
\end{equation}
and
\begin{multline}
\forall  \text{ nonempty }G  \in  \gd(E),\\
\Big(\forall \text{open } U,\; \; U \cap G \neq \emptyset \ \Rightarrow \ \ol{U\cap G} \notin I \Big) \implies  \Big(\;G^* \text{ is nonmeager in } \F\Big) .
\label{c3}%
\end{multline}
\end{prop} 

\begin{proof} Let (\ref{1}) hold. The forward implication of (\ref{char}) is immediate because compact sets are $\gd$; the converse is immediate in light of property $(*)$.
To prove (\ref{c3}), suppose that  $G$ is a nonempty $\gd$ set such that for every nonempty relatively open subset of $G$ has closure not in $I$. By (\ref{1}), to show that  $G^*$ is nonmeager in $\F$, it suffices to show that $G$ cannot be covered by countably many sets in $I$. So suppose $G \subseteq \bigcup_n F_n$ for some  $F_n \in I$. Since $G$ is nonempty and Polish and each $F_n \cap G$ is relatively closed in $G$, for some $n \in \bbn$, $F_n \cap G$ contains a nonempty relatively open subset, say $U$, of $G$. Now $\ol{U} \subseteq \ol{F_n \cap G} \subseteq F_n \in I$, a contradiction. 

Now let (\ref{char}) and (\ref{c3}) hold. The implication from left to right in (\ref{1}) is immediate because if $G \subseteq \bigcup_{n} K_n$ then $G^* \subseteq \bigcup_{n} {K_n}^*$. To prove the other direction, suppose $G$ is a $\gd$ set not covered by the countable union of any sets in $I$. Set $G_0 = G$ and recursively define for successor ordinals $\alpha+1$ and limit ordinals $\lambda$ the sets
\begin{eqnarray*}
G_{\alpha +1} &=& G_{\alpha} \setminus \bigcup \big\{ U : U \subseteq G_{\alpha} \textrm{ relatively open in }G_{\alpha}\textrm{ and } \ol{U} \in I \big\};\\
G_{\lambda} &= &\bigcap_{\alpha < \lambda} G_{\alpha}.
\end{eqnarray*}
For each ordinal $\alpha$, $G_{\alpha}$ is a closed subset of $G$. So for some $\alpha_0 < \omega_1, G_{\alpha_0} = G_{\alpha_0 +1}$. If $G_{\alpha_0}$ were empty, then $G$ would be covered by a countable number of sets $U$ with $\ol{U} \in I$, a contradiction. (Countability can be obtained by considering open sets only from some countable basis.) So $G_{\alpha_0}$ is nonempty, and the closure of every relatively open nonempty subset of it is outside $I$. Now by (\ref{c3}), $G_{\alpha_0}^*$ is nonmeager in  $\F$ and therefore so is $G^*.$
\end{proof}

We will say that a  subset of $E$ is ``everywhere big with respect to $I$'' if all of its nonempty relatively open subsets have closure not in $I$. We make a remark for later use.  The transfinite process described in the proof above can be carried out on any closed subset  $F$  of $E$. At each stage, the set we are removing from $F$ can be written as the countable union of closed sets in $I$. So we can write $F$ as 
\[F = F' \cup \bigcup_{n\in \bbn} {F_n},\]
where each $F_n$ is a compact set in $I$, and $F'$ is a closed subset of $F$,  which, if nonempty, is everywhere big with respect to $I$. 

In light of   Proposition \ref{prop}, to prove Theorem \ref{yay}, we may find  a compact $\F \subseteq \K(E)$ that  serves as a representing set  for $I$ and then simply show \eqref{c3}. To examine the meagerness in $\F$ of some subset $\A$ of $\F$, we will employ the  Banach-Mazur game  in $\F$ on $\A$, which proceeds thus. Players I and II take turns playing nonempty open subsets of $\F$ as follows: 

\begin{tabular}{c c c c c c }
	Player I& $\V_0$ &&$\V_1$& \ldots&\\
	Player II&& $\W_0$ &&$\W_1$&\ldots
\end{tabular}

\noindent satisfying $\V_0 \supseteq \W_0 \supseteq \V_1 \supseteq \W_1 \supseteq \ldots$. Player II wins this run of the game if $\bigcap_n \W_n (=\bigcap_n \V_n) \subseteq \A$. The key fact about this game is that Player II has a winning strategy if and only if $\A$ is comeager in $\F$; see, for example, Section 8.H of \cite{DST} for a proof.  

We will make use of the  Hausdorff metric for the topology on $\K(E)$, derived from some fixed complete metric on $E$ thus. For $x \in E$, let $B(x,r)$ denote the open ball of radius $r$ around $x$. The Hausdorff distance between nonempty compact sets  $F$ and $K$ is then the infimum of all $\epsilon$ such that $F \subseteq \bigcup_{x\in K} B(x, \epsilon)$ and  $K \subseteq  \bigcup_{x\in F} B(x, \epsilon)$.

Finally, we shall make use of the following equivalence proved by Solecki in \cite[section 5]{SS1}. Solecki has shown that for a nonempty set $I \subseteq \K(E)$,  $I$ is a $\gd$ set with property~$(*)$ if and only if there exists a sequence of open, downward closed sets $\U_n \subseteq \K(E)$, $n \in \bbn$, such that $I = \bigcap_n \U_n$ and the sets $\U_n$ satisfy the condition that
		\begin{equation} \label{dcs}
		\forall K \in \U_n \; \; \exists m \in \bbn\;  \forall L \in \U_m\; \; K \cup L \in \U_n.
		\end{equation}
By taking the obvious finite intersections, we may in fact assume that $(\U_n)$ is a decreasing sequence.  For many $\gd$ ideals with property $(*)$, the sets  $\U_n \subseteq \K(E)$ refered to here present themselves naturally. For example, for the ideal of closed null sets for an outer regular Borel measure $\mu$,  we may define the sets $\U_n$  by setting
\[\U_n = \Big\{K \in \K(E) : K \subseteq U \textrm{ for some open }  U  \textrm{ such that }\mu(U) < \frac{1}{n}\Big\}.\]
It is straightforward to see that the sets $\U_n$ satisfy (\ref{dcs}) and their intersection is $I$.

\section{Proving the theorem}
The principal part of the proof of  Theorem \ref{yay} is the construction of a suitable closed and upward closed set $\F$ for the case where $I$ contains only meager sets. This construction requires $E$ to have infinitely many limit points, and also requires the existence of at least one meager compact set with at least three points that is everywhere big with respect to $I$. This necessitates the treatment of two relatively trivial special cases in a separate proposition, which we shall dispense with:

\begin{prop} \label{special}
Let $E$ be a compact Polish space  and let $I \subseteq \K(E)$ be a  $\gd$ ideal of compact sets, with property~$(*)$,  containing only meager sets.  If either one of the following conditions holds:
\begin{enumerate}
\item $E$ has only finitely many limit points, \emph{or}
\item every meager compact subset  of $E$ that is everywhere big with respect to $I$ has at most two points,
\end{enumerate}
then the conclusion of Theorem \ref{yay} holds.
\end{prop}

\begin{proof} Since $I$ has only meager sets,  $I$ has some closed and upward closed set representing set  $\F \subseteq \K(E)$.  Let $G \subseteq E$ be a $\gd$ set that is everywhere big with respect to $I$; we  show that $G^*$ is nonmeager in $\F$. 

Suppose that $E$ has only finitely many limit points. If $G$ contains any isolated point, then $G$ has nonempty interior and so $G^*$ is nonmeager in $\F$. On the other hand, if $G$ contains no isolated point of $E$, it must be a finite, and therefore closed, set (and it is nonempty). So we have $G=\ol{G} \notin I$, and  again $G^*$ is nonmeager in $\F$.

Coming to the other condition, suppose that every meager compact subset  of $E$ that is everywhere big with respect to $I$ has at most two points.  If $x \in E$ belongs to some such finite everywhere big meager set, then clearly $\{x\}$ is not in $I$. Note that  there can be at most two points altogether, say $x_1$ and $x_2$, that show up in such sets. (If there were three distinct points making an appearance in finite everywhere big meager sets, then the three points together would comprise another such set.) 

Take now our nonempty $\gd$ set $G$ which is everywhere big with respect to $I$. If $G$ contains one of the  $x_i$, then $\{x_i\}^* \subseteq G^*$, and so $G^*$ is not meager in $\F$. On the other hand, if $G$ contains no such $x_i$, then by replacing $G$ with a suitable nonempty relatively open subset of itself if necessary, we may assume that $\ol{G}$ contains no such $x_i$ either.  Now write $\ol{G}$ as
\[  \ol{G} =G  \cup  \bigcup_{n\in \bbn} F_n ,\]
where each $F_n$ is a relatively closed and relatively meager subset of $\ol{G}$. Each $F_n$ is also a closed meager  subset of $E$, and by the remark following Proposition \ref{prop}, can be written as
\[F_n = F'_n \cup \bigcup_m {F^n_m},\]
where $F'_n$ is either empty or a meager compact subset of $E$ that is everywhere big with respect to $I$, and each ${F^n_m}$ is in $I$. Since $F'_n$ is disjoint from $\{x_1, x_2\}$, it must simply be empty. We now have 
\[  \ol{G} =G  \cup  \bigcup_{n,m}{F^n_m},\]
and so
\[  {\ol{G}}^* =G^*  \cup  \bigcup_{n,m} {{F^n_m}}^*.\]
Since each ${{F^n_m}}^*$ is meager in $\F$, but ${\ol{G}}^*$ is nonmeager in $\F$, it must be the case that $G^*$ is nonmeager in $\F$.
\end{proof}

Note that the preceding proposition has covered the case of the ideal of meager compact sets (denoted $MGR(E)$), because this ideal satisfies the second condition: there are no nonempty meager compact subsets of $E$ that are everywhere big with respect to $I$.

We now turn to the  construction of a set $\F$ as in Theorem \ref{yay} in the case where $I$ has only meager set. The construction has two stages. First, we construct suitable representing sequences of closed and upwards closed sets for $I$. Second, we join these countably many sets into a single $\F$ that will satisfy the conclusion of Theorem \ref{yay}.
For ease of exposition we separate the first stage into a  proposition of its own. But first, a lemma:

\begin{lemma} \label{bigseq} Let $F$ be a meager compact subset of a Polish space $E$. Then we may find a sequence of pairwise disjoint regular closed sets $(F_n)$  converging to and disjoint from $F$. 
\end{lemma}
\begin{proof}
Since $F$ is meager, for each $\epsilon > 0$ we can find a finite nonempty set disjoint from $F$ whose Hausdorff distance from $F$ is less than $\epsilon$, and of course positive.  This allows us to successively construct a sequence of pairwise disjoint finite sets $F_n$ converging to and disjoint from $F$. By successively replacing each finite set $F_n$ with a set of the form $\ol{\bigcup_{x\in F_n} B(x,\delta_n)}$ for a suitably chosen positive $\delta_n$, we obtain  a sequence satisfying the required conditions. 
\end{proof}

\begin{prop} \label{main}
Let $E$ be a compact Polish space and let $I \subseteq \K(E)$ be a  $\gd$ ideal of compact sets, with property~$(*)$, containing only meager sets. Let $U$, $V$ be open subsets of $E$ with disjoint closures and suppose that $V$ contains a compact meager set $M$ that is everywhere big with respect to $I$. Then there exist  closed and upward closed sets  $\F_m \subseteq \K(E)$, $m \in \bbn$, such that: 
\begin{equation}
\forall  K \in  \K(U),\; K \notin I \implies \exists m \; \F_m \subseteq K^*;
\label{goal1}
\end{equation}
\begin{equation}
\forall  K \in  \K(E),\;	K \in I \implies \forall m \; K^*  \textrm{ is meager in } \F_m; 
\label{goal2}
\end{equation}
\begin{multline}
\forall  \text{ nonempty }G  \in  \gd(U),\\
\Big(
G \text{ is everywhere big w.r.t.}\ I 
\Big) 
\Rightarrow  \Big(\exists m\;G^* \text{ is comeager in } \F_m \Big).
\label{goal3}
\end{multline}
Additionally, 
\begin{equation}
\forall m \in \bbn, \;\F_m \subseteq  \ol{V}^*; \text{ and }
\label{location}
\end{equation}
\begin{equation}
\label{location2}
\forall m \in \bbn, \; \forall \text{open } W \subseteq E, \; \ol{V} \cup \ol{U} \subseteq W \implies  \K(W) \cap  \F_m \neq \emptyset.
\end{equation}
\end{prop}

\begin{proof} 
Fix a decreasing sequence of open, downward closed sets $\U_n \subseteq \K(E)$ 
such that
\[I = \bigcap_n \U_n\]
and (\ref{dcs}) holds for the sets $\U_n$. For $K \in \K(E)$ and $n \in \bbn$, we will say that ``$K$ is $n$-small'' if $K \in \U_n$, else, we will say that ``$K$ is $n$-big''. Clearly $K \in I$ if and only if $K$ is $n$-small for all $n$, and as each $\U_n$ is downward closed, an $n$-big set cannot be contained in an $n$-small set. 

Fix a countable basis $\B$ for $E$, closed under finite unions.  By Lemma \ref{bigseq}, inside $V$ we may find a sequence of pairwise disjoint regular closed sets $(M_i)$, say, converging to and disjoint from $M$.  Note that the union of finitely many members of the sequence $(M_i)$ will have an open superset that is disjoint from all other members of the sequence. With each $B \in \B$ we associate a subsequence of $(M_i)$, whose members we shall denote $M(B,n)$, $n \in \bbn$, in such a way that for each $i$ the set $M_i$ appears in the sequence $(M(B,n))_{n}$ for exactly one set $B \in \B$. 

For $F \in \K(E)$, we introduce some useful terms. For $B\in \B$, we will say that ``$F$ allows $B$'' if $F \cap M(B,n)$ has nonempty interior  for every $n \in \bbn$. For $i \in \bbn$ we will say that ``$F$ is blank in $M_i$'' if $F \cap M_i = \emptyset$. We will  say that ``$F$ has $k$ blanks before $M_i$'' if $k$ of the sets $F \cap M_1$, $F \cap M_2$, $\ldots$, $F \cap M_{i-1}$ are empty.

Now for $m \in \bbn$, we will define a set $\A_{m} \subseteq \K(E)$ thus. For  $F \in \K(E)$, we will say that $F \in \A_{m}$ if the following conditions hold:
\begin{enumerate}[{Cond.} 1.]
	\item  There exists $r \in \bbn$ such that $F \cap M_i$ has nonempty interior for all $i \ge r$.  \label{condtail}
	\item The set  $F \cap M$ is not in $I$. \label{biglim}
	\item There exist $ p \in \bbn$ and $B_{1}, B_{2}, \ldots, B_{p} \in \B$, arranged in the order in which the first terms of their associated subsequences occur within the sequence $(M_i)$ (so that, for example, $M(B_1,1)$ occurs in $(M_i)$ earlier than  $M(B_2,1)$ does) such that: \label{def2}
	\begin{enumerate}[(\ref{def2}a)]
		\item $F$ allows each $B_{i}$; \label{conda}
		\item the union of the sets $\ol{B_{i}}$ is $m$-small, i.e., 
		\[ \bigcup^p_1  \ol{B_i} \in \U_m;\] \label{condb}
		\item for each $j = 1, \ldots, p$, if $F$ has $k_j$ blanks before $M(B_j,1)$ then 
		\[ \bigcup^p_{i=j}  \ol{B_i} \; \in \; \U_{k_j};\] \label{condc}
		\item finally, \[\ol{E \setminus F} \cap \ol{U} \; \subseteq \; \bigcup^p_{i=1} B_i.\] \label{lastcond}
	\end{enumerate}
\end{enumerate}

We will say that the tuple $\langle B_1, B_2, \ldots, B_p, r \rangle$ witnesses the membership of $F$ in $\A_{m}$. Note that Cond.~\ref{condtail} implies that $F$ allows all but finitely many basic sets. Note also that from Cond.~\ref{lastcond} we have that $F$ and the sets $B_i$ cover $\ol{U}$.

It is easy to see that  $\A_{m}$ is upward closed. (If some set $F$ in $\A_{m}$ has witness $\langle B_1, B_2, \ldots, B_p, r \rangle$, then the same witness will work for any compact superset of $F$. Here we make use of the fact that the sequence $(\U_n)$ is decreasing.)

Set $\F_{m} = \ol{\A_{m}}$. This is still  an upward closed set. 

Our goals are now to show that (\ref{goal1}) -- (\ref{location2}) hold. We start with the last two. Certainly (\ref{location}) is immediate (each $F$ in $\F_m$ intersects $M$).  To see (\ref{location2}), fix  $m \in \bbn$ and fix an open superset $W$ of $\ol{V} \cup \ol{U}$. We want to show that  $\K(W) \cap  \F_m \neq \emptyset$. Let $W'$ be  open  such that $\ol{V} \cup \ol{U} \subseteq W' \subseteq \ol{W'}\subseteq W.$ It is clear that the set $\ol{W'}$ satisfies  
Cond.~ \ref{condtail} and Cond.~ \ref{biglim}, and that $\ol{W'}$ allows all basic sets and has no blanks at all in ${(M_i)}$. Since 
\[\ol{E \setminus \ol{W'}} \cap \ol{U} = \emptyset,\] 
we can simply take any basic set $B$ whose closure is in $\U_m$ to get a witness $\langle B, 1 \rangle$ for the membership of $\ol{W'}$ in $\A_m$. (Certainly $\U_m$ will  contain the closures of some basic sets, just because any member of the   open set  $\U_m$ is the limit of a sequence of  the closures of basic sets.)

We move on to   (\ref{goal1}). Fix a  compact set $K \subseteq U$ not in $I$, and fix  $m_0$ such that $K \notin \U_{m_0}$. For $F \in \A_{m_0}$, we have $\ol{E \setminus F} \cap \ol{U} \subseteq \bigcup^p_1 B_i$ for some basic sets $B_i$ satisfying the condition that  $\bigcup^p_1  \ol{B_{i}}$ is $m_0$-small.  The sets $F$ and $\bigcup^p_1  \ol{B_{i}}$ cover $\ol{U}$ and therefore cover $K$. Since $K$ is $m_0$-big it cannot fit inside $\bigcup^p_1  \ol{B_{i}}$, and so must intersect $F$. Thus $\A_{m_0} \subseteq K^*$; the latter being closed we also get 
$\F_{m_0} \subseteq K^*$.

Next, (\ref{goal2}). Fix  $K \in I$,  $m \in\bbn$ and $\epsilon > 0$.  Let $F \in \A_{m}$ and let  the tuple $\langle B_1, B_2, \ldots, B_p, r \rangle$ witness the membership of $F$ in $\A_{m}$. We will find a set in  $\A_{m} \setminus K^*$  within $\epsilon $ of $F$; this will suffice to show that $K^*$  is nowhere dense, and hence meager, in  $\F_{m}$.

The sets $B_i$ satisfy the condition that  
\[\bigcup^p_{i=1}  \ol{B_{i}} \in \U_m,\] 
and they also satisfy the following $p$ many conditions: for each $j = 1, \ldots, p$, if $F$ has $k_j$ blanks before $M(B_j, 1)$ then 
\[ \bigcup^p_{i=j}  \ol{B_i} \; \in \; \U_{k_j}.\] 

We may now pick a $q \in \bbn$ large enough so that:
\begin{itemize}
\item we can add a $q$-small set to $\bigcup^p_1  \ol{B_{i}}$ and stay in $\U_m$;
\item for each $j = 1, \ldots, p$, we can add a $q$-small set to $\bigcup^p_{j} \ol{B_i}$ and stay in $\U_{k_j}$;
\item if $F$ has a total of $t$ blanks in the sequence $(M_i)$, then $q > t$.
\end{itemize}

Now, $F\cap M \nsubseteq K$ because $F\cap M \notin I$. Let  $ x\in F\cap M\setminus K$. Pick a positive $\delta < \epsilon$ such that $\ol{B(x, \delta)}$ is disjoint from $K$. Since  $M$ is in  $B(x, \delta)^*$, which is an open subset of $\K(E)$, we may fix $i_0 \in \bbn$ such that $\forall i \ge i_0, \; M_i \in B(x, \delta)^*$, i.e., $M_i \cap B(x, \delta) \neq \emptyset$. As the interior of each  $M_i$ is dense in it, we have $\forall i \ge i_0, \; (M_i)^\circ \cap B(x, \delta) \neq \emptyset$.

Now consider the finitely many sets $M_i$ for  $i < i_0$. For each of these sets $M_i$ we  define an open subset $V_i$ thus. If $F \cap M_i \neq \phi$, then, recalling that $M_i$ is regular and $K$ is meager, we  fix a nonempty open set $V_i \subseteq M_i$  such that $\ol{V_i} \subseteq \bigcup_{y\in F} B(y,\delta)$ and $\ol{V_i} \cap K = \emptyset$. 
If on the other hand $F \cap M_i = \phi$, then we simply set $V_i = \emptyset$.

Now consider a sequence of basic open sets containing $K$ whose closures converge to $K$. As $K$ belongs to the open set $\U_q$, these closures eventually lie in $\U_q$. Pick a member $B$ of this sequence such that:
\begin{itemize}
\item  $M(B,1)$ occurs in the sequence $(M_i)$ \emph{after} $M_{i_0}$ and \emph{after} $M(B_p, 1)$, and
\item $\ol{B} \in \U_q$.
\end{itemize}
Now let $B'$ be an open set such that $K \subseteq B'\subseteq \ol{B'} \subseteq B$, and set \[F' = (F \setminus B') \ \cup\  \ol{B(x, \delta)} \  \cup \ \bigcup_{i<i_0}\ol{V_i},\] 
which is clearly disjoint from $K$. We now claim that $F'$ is in $\A_m$. 

Cond. \ref{condtail} is satisfied as $\ol{B(x, \delta)} \cap M_i $ has nonempty interior for all $i \ge i_0$. Cond.~\ref{biglim}  is satisfied as $\ol{B(x, \delta)} \cap M$  contains  $\ol{B(x, \delta) \cap M}$, which is not in $I$ as $M$ is everywhere big with respect to $I$. 

The tuple $\langle B_1, B_2, \ldots, B_p, B, i_0 \rangle$ witnesses the membership of $F'$ in $\A_{m}$. To see this, first note that $F' \cap M_i$ has nonempty interior for every $i$ where $F \cap M_i$ had nonempty interior, so  $F'$ allows any basic set that $F$ allowed, including each $B_j$. The basic set $B$ is allowed as the entire sequence $(M(B,n))_n$ lies  in the $i_0$-tail of $(M_i)$. Hence Cond.~(\ref{conda}) holds.  Cond.~(\ref{condb}) holds for this witness by choice of $q$. Cond.~(\ref{condc}) now consists of $p+1$ conditions. For $j=1, \ldots, p$, if $F'$ has $k'_j$ blanks before $M(B_j, 1)$, then 
$k'_j \le k_j$, and by choice of $q$ we have 
 $\bigcup^p_j  \ol{B_i} \cup \ol{B} \in \U_{k_j} \subseteq \U_{k'_j}$. 
The $p+1$st condition is that we must have $\ol{B} \in \U_{k}$, where $k$ is the number of blanks that $F'$ has before $M(B,1)$. Now, $F'$ can have no more than $t$ blanks before $M(B,1)$ because it has no more than $t$ blanks anywhere, and thus  $\ol{B} \in \U_q \subseteq \U_t \subseteq \U_{k}$. So Cond.~(\ref{condc}) holds. Finally, for Cond.~(\ref{lastcond}) note that, by definition of $F'$, we have $F \setminus B' \subseteq F'$, and thus $E \setminus F' \subseteq E \setminus (F \setminus B' ) = B' \cup E \setminus F$. So
 \[  \ol{U} \cap  \ol{E \setminus F'} \   \subseteq \ 
  \ol{U}  \cap \Big[\ol{B'} \cup \ol{E \setminus F} \Big] 
\   \subseteq  \ 
 \ol{B'} \cup \bigcup^p_{i=1}  {B_{i}} 
 \   \subseteq \   B \cup \bigcup^p_{i=1}  {B_{i}}.\] 
  
So we have $F' \in \A_m \setminus K^*$. It remains to examine the distance of $F'$ from  $F$. We first note that the set $F \cup \ol{B(x, \delta)} \cup \ \bigcup_{i<i_0}\ol{V_i}$ is within $\epsilon$ of $F$. (This follows from the fact that  $x \in F$ and $\delta < \epsilon$, and by the choice of $V_i$'s.) In forming the set $F'$  we have removed some points from this set  and may have obtained a set that is more than $\epsilon$ away from $F$. Recalling again that $K$ is meager and $\A_m$ is upward closed, this can be remedied by adding to $F'$ a finite set of points, all outside $K$, to obtain a set in $\A_m \setminus K^*$ once again within $\epsilon$ of the original set $F$.

Thus we have shown that $K^*$ is meager in  $\F_m$. 

Having establised that the sets $\F_m$ determine membership of compact sets $K$ in $I$,  we now address  (\ref{goal3}):

Let $G \subseteq U$ be a nonempty $\gd$ set that is everywhere big with respect to $I$. Let $(H_n)$ be an increasing sequence of closed sets that are relatively meager in $\ol{G}$ such that  $G = \ol{G} \setminus \bigcup_n H_n$. 

Fix $m_0$ such that  $\ol{G} \notin \U_{m_0}$, so that $\F_{m_0} \subseteq \ol{G}^*$.  In order to show that $G^*$ is comeager in $\F_{m_0}$, we will play the Banach-Mazur game in $\F_{m_0}$, on the set $G^*$; we will describe a winning strategy for Player II, i.e., if Player I is playing sets $\V_n$ and Player II is playing sets $\W_n$, (all open subsets of $\F_{m_0}$ satisfying the inclusions of the game) we will show that $\bigcap_n \W_n \subseteq G^*$.

Player I starts the game by playing some $\V_1$.  If $F$ is a set in   $\V_1 \cap \A_{m_0}$ with witness $\langle B_1, \;B_2,\; \ldots B_p, r\rangle$,  then we know that $\ol{G}\setminus \bigcup^p_1  \ol{B_i}$
is nonempty. (In fact, this is why we know that $\ol{G}$ intersects $F$.)  If Player II can construct $\W_1$, a further nonempty open subset of $\F_{m_0}$, and $D_1$, a nonempty relatively open subset of $\ol{G}$, such that
\begin{itemize}
	\item $\ol{D_1} \cap {H_1} = \emptyset$, and
	\item  $\W_1  \subseteq \ol{D_1}^*$,
\end{itemize}
and if Player II can keep going in this way, producing with each successive $\W_n$ that it plays, an open 
 subset ${D_n}$ of $\ol{G}$, such that
 \begin{itemize}
	\item $D_n \subseteq D_{n-1}$ for $n>1$,  
	\item $\ol{D_n} \cap {H_n} = \emptyset$, and
	\item   $\W_n \subseteq \ol{D_n}^*$,
	\end{itemize}
then Player II will have a winning strategy. 

Therefore it suffices to show the following.

\textbf{Claim:} 
Let $D$  and  $H$  be nonempty  subsets of $\ol{G}$, with $D$ relatively open in $\ol{G}$, and $H$ relatively meager and  closed in $\ol{G}$. Suppose that  $\V$ is  a nonempty relatively open subset of $\F_{m_0}$ such that for any $F \in \V \cap \A_{m_0}$, if  $\langle  B_1, B_2, \ldots, B_p, r \rangle$ witnesses the membership of $F$ in $\A_{m_0}$, then 
	\[\ol{D}\setminus {\bigcup_{j=1}^p \ol{B_{j}}} \neq \emptyset.\]
Then there exists a further nonempty relatively open subset $\V'$ of $\F_{m_0}$, contained in $\V$, and  a further nonempty relatively open subset $D'$ of $\ol{G}$, such that $\ol{D'} \subseteq D \setminus H$ and,  for any $F \in \V' \cap \A_{m_0}$, if  $\langle  B_1, B_2, \ldots, B_p, r \rangle$ witnesses the membership of $F$ in $\A_{m_0}$, then 
	\[\ol{D'}\setminus {\bigcup_{j=1}^p \ol{B_{j}}} \neq \emptyset.\]
(This then implies that $F \in \ol{D'}^* $, and thus we have  $\V'  \subseteq \ol{D'}^*$.)

We prove the claim.

Without loss of generality, we may assume that $\V$ has the form 
\[\V= \{K \in \F_{m_0}: K \subseteq V_0,\; K \cap V_i \neq \emptyset \ \;\forall i=1,\ldots, l\},\] where the sets $V_1, \ldots, V_l$ are nonempty open subsets of the open set $V_0$. We can further assume that each $V_i$ is an $\epsilon$-ball for some fixed $\epsilon$.

Fix $F \in \V \cap \A_{m_0}$, with witness $\langle  B_1, B_2, \ldots, B_p, r_F \rangle$, say. Since the open set 
\[E\setminus {\bigcup_{j=1}^p \ol{B_{j}}} \] 
intersects $\ol{D}$, it intersects ${D}$. As $H$ is meager in $\ol{G}$, within ${D}$ we can find a further open subset of $\ol{G}$, say $D'$, such that 
\[\ol{D'} \; \subseteq \; D \setminus \Big(\bigcup_{j=1}^p \ol{B_{j}} \cup H \Big).\]
Since ${G}$ is everywhere big with respect to $I$, we know that ${\ol{D'}} \notin I.$ Fix a $k$ such that $\ol{D'} \notin \U_{k}$.

For each  $j=1, \ldots, p$, we have that $F$ allows $B_j$. In the remainder of the proof, the idea is to find an open neighbourhood of $F$ in $\K(E)$ within which the specific sets $B_1, \ldots, B_p$ remain allowed, but basic sets $B$ other than these are not allowed unless the closure of their union is $k$-small. This will force any member of $\A_{m_0}$ in this neighbourhood to intersect $\ol{D'}$, which is $k$-big. 

Recall that $B_p$ is the last basic set in the witnessing tuple. Pick $r$ such that 
\begin{itemize}
	\item $M_r$ appears in $(M(B_p, n))_n$, the subsequence of $(M_i)$ associated with $B_p$;
	\item $r$ is large enough so that for  $i\ge r$, the Hausdorff distance between $M_i$ and $M$ (the limit of $(M_i)$) is less than $\epsilon$. 
\end{itemize} 

There are up to $r$ many  basic sets $B$ whose associated subsequences have appeared within the initial $r$ terms of the sequence $(M_i)$. These include the sets $B_1, \ldots, B_p$. For each of these finitely many basic sets $B$ \emph{other than} $B_1, \ldots, B_p$, find one member of the associated sequence that occurs \emph{after} $M_r$ and tag it. Now again find  $r' \in \bbn$ such that $M_{r'}$ lies within the subsequence $(M(B_p, n))_n$ and also $M_{r'}$ occurs \emph{after} all the up to $r-p$ tagged sets. Ensure also, by choosing a larger $r'$ if necessary, that there are at least $k$ sets $M_i$ with $r<i<r'$ that do not occur in $(M(B_j, n))_n$  for any  $j=1, \ldots, p $. Now let $R$ be the union of all the sets $M_i$ for $r<i<r'$ \emph{except} for those that are associated with the specific basic sets $B_1, \ldots, B_p$., i.e.,
\[R = \bigcup \big\{ M_i :  r <i<r', M_i \text{ does not occur in }(M(B_j, n))_n \text{ for  } 1\le j \le p\big\}.\]  
$R$ is a closed set that contains all the tagged sets and it is a finite union of at least $k$ many members of $(M_i)$. It is disjoint from all members of $(M_i)$ that it does not contain, and also from $M$.  We also know that $R$ cannot completely contain any $\epsilon$-ball. (Any such ball that it did contain would be forced to intersect $M$ as well as be disjoint from $M$.)   In particular, $R$ cannot completely contain any of the sets $V_i$. For $1 \le i \le l$, fix $y_i$ in $V_i \setminus R$.

Set \[\V' = \K( V_0 \setminus R) \cap \V.\] 
We show that the sets $D'$, $\V'$ satisfy all the required conditions. We already have $\ol{D'} \subseteq D \setminus  H$. 
We now show that $\V' \neq \emptyset $. Since $F \in \V$ we have that  $F \subseteq V_0$. 
Let $O$ be an open superset of $R$ so chosen that it does not contain any of the points $y_i$ for $i = 1, \ldots l$ and also so that it does not intersect any set $M_j$ that is disjoint from $R$.  We may also ensure that $\ol{O} \cap \ol{U} = \emptyset$, as $R$ is a closed set disjoint from $\ol{U}$. Now let 
\[F'=(F \setminus O) \cup \{ y_i: i=1, \ldots l\}.\]
It is clear that $F' \subseteq V_0 \setminus R$ and $F'\cap V_i \neq \emptyset$ for $1 \le i \le l$ . To show that $F' \in \A_{m_0}$, it suffices to show that $F \setminus O \in \A_{m_0}$, since the addition of points does not affect membership in  $\A_{m_0}$. First note that since $O$ meets only finitely many of the sets $M_i$, the sequence $ {(M_i \cap (F \setminus O))}_i$ still has a tail with nonempty interiors (starting at say $r_{F'}$), so that Cond.~\ref{condtail} holds. Next, note that $F \setminus O$ allows each of $B_1, \ldots, B_p$, which already satisfy Cond.~(\ref{condb}) and Cond.~(\ref{condc}), and that 
\[\ol{E \setminus (F\setminus O)} \cap \ol{U} \ = \ \ol{E \setminus F} \cap \ol{U}.\]
Therefore the tuple $\langle  B_1, B_2, \ldots, B_p, r_{F'} \rangle$ witnesses the membership of $F\setminus O$ in $\A_{m_0}$.
Thus $\V'$ is nonempty. 

We now show that  for any $L \in \V' \cap \A_{m_0}$, if the membership of $L$ in $\A_{m_0}$ is witnessed by $\langle  B'_1, B'_2, \ldots, B'_{p'}, r_L \rangle$, then 
	\[\ol{D'}\setminus {\bigcup_{j=1}^{p'} \ol{B'_{j}}} \neq \emptyset.\]
	 So let $L \in \V' \cap \A_{m_0}$. What are the basic sets $B$ that $L$ can allow? (This is the d\'enouement!) If $B$ is not one of the original $B_i$ obtained from the set $F$, then $M(B,1)$ cannot occur before $M_r$. (If it did, some member of its associated sequence is a tagged set included in $R$, but $L$ is disjoint from $R$.) $M(B,1)$  cannot occur between $M_r$ and $M_{r'}$ as in this range the only $M_i$ remaining in $V_0 \setminus R$ were part of the sequences associated with the original $B_i$. Therefore the  sequence associated with $B$ is entirely contained in the $r'$-tail of $(M_i)$. 
	
	With this in mind, any witness of the membership of $L$ in $\A_{m_0}$ must have the form $\langle  B_{i_1}, B_{i_2}, \ldots, B_{i_s}, B'_1, B'_2, \ldots, B'_{s'}, r_L \rangle$, where the sets $B_{i_1}, B_{i_2}, \ldots, B_{i_s}$ are taken from $B_{1}, B_{2}, \ldots, B_{p}$ and  $M(B'_1,1)$ lies in the $r'$-tail of $(M_i)$. 
	
	Note also that $L$, missing the whole of $R$, has at least $k$ many blanks occuring before $M(B'_1,1)$ in the sequence $(M_i)$. Therefore by Cond.~(\ref{condc}),
	\[\bigcup_{j=1}^{s'} \ol{B'_{j}} \; \in \; \U_{k},\]
	while $\ol{D'} \notin \U_k$. We also know that the whole of $\ol{D'}$ lies outside  $\bigcup_{j=1}^p \ol{B_{j}}$. So 
	\[\ol{D'}\setminus \Big( \bigcup_{j=1}^{s} \ol{B_{i_j}} \; \cup \; \bigcup_{j=1}^{s'} \ol{B'_{j}} \Big) \; = \; \ol{D'}\setminus \bigcup_{j=1}^{s'} \ol{B'_{j}} \;\neq \; \emptyset.\]

Thus the sets $\V'$ and $D'$  satisfy all the required conditions, and the Claim holds, completing the proof. 
\end{proof}

We now wish to construct a single upward closed set $\F$ satisfying Theorem \ref{yay} for ideals that contain only meager sets. We first  find pairs of disjoint open sets $U_n, V_n$ such that the sets $U_n$ cover $E$ and for each pair $U_n, V_n$, we can apply  Proposition \ref{main} to get  representing sequences ${(\F^n_m)}_m$. We then want to obtain a single representing set $\F$ that does the job for the whole space $E$. The idea is that we form $\F$ by taking some manner of union of all the  $\F^n_m$, in such a way that for a fixed $n,m$  we are able to find an open subset of $\K(E)$ within which all the members of $\F$ come only from this particular $\F^n_m$. The essence of the construction of $\F$ is that we marry each $\F^n_m$  to a particular `permission set' before throwing them all into a common union, in such a way that permission sets can be excluded or included from our desired open set as needed.

We proceed. 

\begin{thm} \label{yay1} 
	Let $E$ be a compact Polish space and let $I \subseteq \K(E)$ be a  $\gd$ ideal of compact sets, with property~$(*)$, containing only meager sets. Then there exists a compact upward closed set $\F \subseteq \K(E)$ such that any $\gd$ subset $G$ of $E$ is covered by countably many sets in $I$ if and only if $G^*$ is meager in $\F$.
	\end{thm}

\begin{proof}   
By Proposition \ref{prop}, it suffices to construct  a  set $\F \subseteq \K(E)$, closed and upward closed, such that 
\begin{equation}
\forall  K \in  \K(E),\; K \notin I \implies K^*  \text{ is nonmeager in } \F;
\label{goal11}
\end{equation}
\begin{equation}
\forall  K \in  \K(E),\;	K \in I \implies K^*  \textrm{ is meager in } \F; 
\label{goal22}
\end{equation}
\begin{multline}
\forall  \text{ nonempty }G  \in  \gd(E),\\
\big(
G \text{ is everywhere big w.r.t.}\ I 
\big) \implies  \big(G^* \text{ is nonmeager in } \F \big).
\label{goal33}
\end{multline}

By Proposition \ref{special}, we may assume that  $E$ has infinitely many limit points, and  $E$ has  some meager compact subset  with at least three points that is everywhere big with respect to $I$. 
We first construct nonempty open subsets $U_1, U_2, U_3$ and $V_1, V_2, V_3$  of $E$ such that 
\begin{enumerate}
\item the sets $U_1, U_2, U_3$  form an open cover of $E$; 
\item the following strict inclusions hold:
\[\phantom{movemo}\ol{V_1}  \subsetneq U_2  \setminus  (\ol{U_3} \cup \ol{U_1}),\; \ol{V_2} \subsetneq U_3 \setminus  (\ol{U_1} \cup \ol{U_2}), \; \ol{V_3} \subsetneq U_1 \setminus  (\ol{U_2} \cup \ol{U_3}),\]
and, in each of the three set inclusions above, if the smaller set is removed from the larger one then at least one limit point of $E$ is left behind; 
\item for each $n=1, 2, 3$, there exists a sequence of closed and upward closed sets ${(\F^n_m)}_m $, such that  the sets $U=U_n$, $V=V_n$ and $\F_m = \F^n_m$ all satisfy (\ref{goal1}),  (\ref{goal2}),  (\ref{goal3}),   (\ref{location}), and  (\ref{location2}).
\end{enumerate}  
To do this, fix a meager compact subset $M$ of $E$, with at least three points, that is everywhere big with respect to $I$. We may assume that $E \setminus M$ contains at least three limit points of $E$. (Why? If  $M$  contains all but perhaps two of the infinitely many limit points of $E$, we can pick a suitable relatively open subset $U$ of $M$ such that $\ol{U}$ contains at least three of these points and omits another three. The set $\ol{U}$ is  still  both  meager and  everywhere big with respect to $I$, so we can replace $M$ with $\ol{U}$.) Let us say that  $a_1$, $a_2$, and $a_3$ are limit points of $E$ not in $M$. 

Since $M$ contains at least three points, we may find open sets $V_1, V_2, V_3$ with disjoint closures, all of which intersect $M$, and none of which contains any point $a_i$. Now let $U_1$ and $U_2$ be open sets with disjoint closures such that 
\[U_1 \supseteq \ol{V_3} \cup \{a_3\}, \; U_2 \supseteq \ol{V_1} \cup \{a_1\}, \;\text{and}\; \big(\ol{U_1} \cup \ol{U_2}\big) \cap\big(\ol{V_2} \cup \{a_2\}\big) = \emptyset,\]
and let $U_3$ be an open set such that 
\[ \ol{U_3} \cap \big(\ol{V_1} \cup \ol{V_3} \cup \{a_1, a_3\} \big)  = \emptyset  \; \text{and}\;  U_1 \cup U_2\cup U_3 = E.\]

Now note that the closure of any nonempty open subset of $M$ retains the properties of being meager and being everywhere big with respect to $I$. So for each $n$, since  $V_n$ intersects $M$,  we may find a  set $M_n  \subseteq V_n$ such that $M_n$ is a meager compact set that is everywhere big with respect to $I$.  For each $n$, the sets $U=U_n$, $V=V_n$, and  $M=M_n$ now satisfy the hypothesis of  Proposition \ref{main}, which we apply to get a sequence ${(\F^n_m)}_m$ as described in that proposition. 

It remains to construct the single set $\F$. Having obtained the sets $U_i$ and $V_i$, now fix open sets $W_1$, $W_2$, $W_3$, each containing a limit point, such that 
 \[
 \ol{W_1} \subseteq U_3 \setminus  (\ol{U_1} \cup \ol{U_2} \cup \ol{V_2}), 
 \;  
 \ol{W_2} \subseteq U_1 \setminus  (\ol{U_2} \cup \ol{U_3}\cup \ol{V_3}), \;
 \ol{W_3} \subseteq U_2  \setminus  (\ol{U_3} \cup \ol{U_1}\cup \ol{V_1}).
 \]

Note that we have ensured that if $\{i,j,k\} = \{1,2,3\}$, the sets $\ol{U_i}$, $\ol{V_i}$, and $\ol{W_i}$ are pairwise disjoint and their union $\ol{U_i} \cup \ol{V_i} \cup \ol{W_i}$ is disjoint from either $\ol{V_j} \cup \ol{W_k}$ or from $\ol{V_k} \cup \ol{W_j}$.

We now build `permission sets' inside the sets $W_n$ thus.  Inside each open set $W_n$ we may find a sequence of nonempty open sets  ${({P^n_i})}_i $ such that the sets $\ol{P^n_i}$ are pairwise disjoint, and for any $j$, it is possible to find an open superset of $\ol{P^n_j}$ that is disjoint from  $\ol{P^n_i}$ for all $i \neq j$. (For example, we may take a convergent  sequence  of distinct points in $W_n$ and put suitable open balls around the points to get the sets $P_i^n$.)

We now define  $\F$. For $n=1,2,3$ and $m\in \bbn$,  define 
$$\B^n_m =  \F^n_m  \cap {P^n_m}^*.$$ 
Since $ \F^n_m$ is upward closed, so is $\B^n_m$. Also, recalling that $ \F^n_m \subseteq V_n^*$ and $ {P^n_m} \subseteq W_n$, we note that any set in $\B^n_m$ must meet both $V_n$ and $W_n$. Now define 
\[ \F = \ol{\bigcup_{n,m} \B^n_m},\]
which is still an upward closed set.

Claim: For any $n \le 3$ and $m\in \bbn$, there exists an open set $O \subseteq E$ such that 
$$\emptyset \; \neq \;  \K(O) \cap \F \; \subseteq \; {\F^{n}_{m}} \cap  {P^n_m}^* .$$

Proof of claim: Fix   $n,m$. Fix an open set $P$ such that $\emptyset \neq \ol{P} \subseteq P^n_m.$ For the appropriate distinct $i$ and $j$, both different from $n$, $\ol{U_n} \cup \ol{V_n} \cup \ol{W_n}$ does not intersect $\ol{W_i} \cup \ol{V_j}$.  Let $O'$ be an open superset of $\ol{U_{n}} \cup \ol{V_{n}} $  disjoint from  $\ol{W_i} \cup \ol{V_j} \cup \ol{W_n}$, and let $O = O' \cup P$. Note that  $O \cap W_n = P$ and so $O$ is disjoint from $\ol{P^{n}_{k}}$ for $k \neq m$. 

To see that this $O$ is as required, first note that (\ref{location2}) holds for the sets $U=U_n$, $V=V_n$ and $\F_m = \F^n_m$, so there exists some $F \in \K(O) \cap \F^n_m$. Fixing any point $x \in P$, the set  $F  \cup \{x\}$ is now a member of $\K(O) \cap \B^{n}_{m}$, so that  $\K(O) \cap \F \neq \emptyset$. 

For the required set inclusion, we will actually show that  
$$ \K(O) \cap \F \; \subseteq \; {\F^{n}_{m}} \cap  {\ol{P}}^*.$$ 

Suppose that  there is some $ F \in \K(O) \cap \F $ that is not in  ${\F^{n}_{m}} \cap  {\ol{P}}^*$. This means that the open set $ \K(O)  \setminus ( \F ^n_m  \cap  {\ol{P}}^*)$ intersects $\F$. This open set must then  contain a set $F'$ in  $\B^{n_1}_{m_1}$ for some $n_1, m_1$ (since $\F = \ol{\bigcup_{i,j} \B^i_j}$).
 Any set in $\B^{n_1}_{m_1}$ meets $\ol{V_{n_1}}$, and  $P^{n_1}_{m_1}$. However, since $F' \subseteq O$, the only possibility is that $n_1 =n$   and $m_1 =m$; all other possibilities have been excluded from $O$. This is a contradiction. So the claim holds. 

Now (\ref{goal11}) is immediate: let $K$ be a compact set not in  $I$. Since $I$ is closed under countable union, and the $\fs$ sets $U_n$ form a cover of $E$, we may simply assume that $K$ is contained in one of the sets $U_n$, say $U_{n_0}$. Take  $m_0$ such that ${\F^{n_0}_{m_0}} \subseteq K^*$. To establish  (\ref{goal11}), use the claim above to find an open set $O$ such that $\emptyset \neq \K(O) \cap \F \subseteq {\F^{n_0}_{m_0}} \subseteq K^*$.

To see (\ref{goal22}), let $K \in I$. We show that $K^*$ is nowhere dense and hence meager in $\F$. Recalling the definition of $\F$, let $F \in \B^n_m$ for some $n,m$, and let $\epsilon > 0$; we show that there is a compact set in $\B^n_m \setminus K^*$ within $\epsilon$ of $F$.

Since  $F \in \F^n_m$ and $K^*$ is meager in $\F^n_m$, there is some $F'$ in $\F^n_m \setminus K^*$ within $\epsilon$ of $F$. 
If $F' \cap P^n_m \neq \emptyset$, we are done. If $F' \cap P^n_m = \emptyset$, the addition to $F'$ of a suitable point in $P^n_m \setminus K$ (using the fact that $K$ is meager and $F$ intersects $P^n_m$) gives a set still within $\epsilon$ of $F$ and in $\B^n_m$, and we are done. 

Finally, we turn to  (\ref{goal33}). Let $G  \in  \gd(E)$ be a set that is everywhere big with respect to $I$. Again, since the sets $U_n$ form a cover of $E$, by replacing $G$ with a suitable relatively open subset of itself, we may  assume that $G$ is contained in one of the sets $U_n$, say $U_{n_0}$. Now (since (\ref{goal3}) holds with $\F_m =\F^{n_0}_m$ and $U=U_{n_0}$)  there exists  $m_0$ such that $G^*$ is comeager in ${\F^{n_0}_{m_0}}$. Take an open set $O \subseteq E$ such that $\emptyset \neq \K(O) \cap \F \subseteq {\F^{n_0}_{m_0}} \cap  {P^{n_0}_{m_0}}^* $. The set $G^*$ is comeager in $\K(O) \cap \F$. To see this precisely, let $\Gee$ be a $\gd$ subset of $\K(E)$ such that $\Gee \cap {\F^{n_0}_{m_0}}$ is dense in ${\F^{n_0}_{m_0}}$ and contained in $G^*$. Now consider the set $\Gee \cap \K(O) \cap \F$, and note that:
\begin{itemize}
\item it  is a $\gd$ subset of $\F$;
\item it is contained in $G^*$ as  anything in $\K(O) \cap \F$ is actually from  ${\F^{n_0}_{m_0}}$;
\item  it is dense in $\K(O) \cap \F$. To see this, let $\U$ be an open set that intersects $\K(O) \cap \F$. Then 
$$\emptyset  \:  \: \neq   \: \: \U \cap \K(O) \cap \F  \:   \: \subseteq   \:  \: \U \cap \K(O) \cap {\F^{n_0}_{m_0}} \cap  {P^{n_0}_{m_0}}^*.$$
The rightmost set being an open subset of $ {\F^{n_0}_{m_0}} $, it must contain a member of $\Gee$, which will automatically be in $\U \cap \K(O) \cap \F$. 
\end{itemize}

 This proves (\ref{goal33}).
\end{proof}

It remains to prove the theorem for ideals that may contain non-meager sets. 

\begin{thm} \label{yay2} 
	Let $E$ be a compact Polish space and let $I \subseteq \K(E)$ be a  $\gd$ ideal of compact sets, with property~$(*)$. Then there exists a  compact set $\F \subseteq \K(E)$ such that any $\gd$ subset $G$ of $E$ is covered by countably many sets in $I$ if and only if $G^*$ is meager in $\F$.
	\end{thm}

\begin{proof}
Starting with the whole space $E$, we once again carry out the transfinite procedure of Proposition \ref{prop} to obtain $E$ as the disjoint union of two sets:
\[E = U \cup E'\]
where  $E'$ is a closed subset of $E$ that satisfies the condition that
\[\forall F \in \K(E'),\; F \in I \implies  F \text{ is relatively meager in } E',\]
and $U$ is an open set that may  be written as 
\[U = \bigcup_n F_n,\]
where each $F_n$ is in $I$. This form of $U$ together with the fact that $I$ is closed under countable union makes it immediate that $\K(U) \subseteq I$. It is also immediate that any compact subset $K$ of $E$ is in $I$ if and only if $K \setminus U$  is in $I$. 

Now let $I' = I \cap \K(E')$. It is easily checked that $I'$ is a $\gd$ ideal with property $(*)$, both when $E$ and when $E'$ are considered as the underlying space. For a set $A\subseteq E'$, let \[A^* = \{K\in \K(E):  K \cap A  \neq \emptyset\}\]
as usual and let
 \[ {A^*}' = \{K\in \K(E'):  K \cap A  \neq \emptyset\}.\]
Since $I'$ contains only meager subsets of $E'$,  the space $E'$ and the ideal $I'$ satisfy the hypothesis of Theorem \ref{yay1}. So there exists a 
nonempty compact set $\F \subseteq \K(E')$  such that 
\begin{multline}
\forall  \text{ nonempty }G  \in  \gd(E'),\\
    \Big( G \subseteq \bigcup_{n\in \bbn} K_n\ \text{ for some } K_n \in I' \Big) \iff \Big({G^*}'  \textrm{ meager in } \F \Big). 
    \label{primefact}
  \end{multline}
We remark  that Theorem \ref{yay1} has given us the upward closedness of  $\F$ only as a subset of $\K(E')$. In any case,  we have
\[\F \subseteq \K(E') \subseteq \K(E).\]
We now claim that:
\begin{multline}
\forall  \text{ nonempty }G  \in  \gd(E),\\
    \Big( G \subseteq \bigcup_{n\in \bbn} K_n\ \text{ for some } K_n \in I \Big) \iff \Big({G^*}  \textrm{ meager in } \F \Big). 
    \label{lastgoal}
  \end{multline}
So let $G$ be a nonempty $\gd$ subset of $E$. Suppose $G^*$ is meager in $\F$. Since ${(G\cap E')^*}' \subseteq G^*$, we must have that ${(G \cap E')^*}'$ is meager in $\F$. So by (\ref{primefact}), 
\[ G \cap E' \subseteq \bigcup_{n\in \bbn} K_n\ \text{ for some } K_n \in I',\] 
and now
\[ G =  (G \cap E') \cup (G \cap U) \ \subseteq \ \bigcup_{n\in \bbn} K_n  \ \cup \ \bigcup_{n\in \bbn} F_n,\]
where each $K_n$ and each $F_n$ is in $I$.

To prove the converse of (\ref{lastgoal}), suppose that $G \subseteq \bigcup_{n} K_n\ \text{ for some } K_n \in I$. For each $n$, since $K_n \in I$ we have $(K_n \cap E') \in I'.$ Now, 
\[G^* \cap \F \ \subseteq \ \bigcup_{n\in \bbn} \big({K_n}^* \cap \F \big) \  = \  \bigcup_{n\in \bbn} \Big(({{K_n\cap E'})^*}' \cap \F\Big).\]
(The last equality comes from the fact that every set in $\F$ is wholly contained in $E'$ already.) Now since $({{K_n\cap E'})^*}'$ is meager in $\F$ for every $n$, we have that $G^*$ is meager in $\F$. This concludes the proof. 
\end{proof}	   

Theorems \ref{yay1} and \ref{yay2} together give Theorem \ref{yay}.

\end{document}